\documentclass[12pt]{article}

\usepackage{amsmath}
\usepackage{amsfonts}
\usepackage{amsthm}
\usepackage{amssymb}
\usepackage[margin=1in]{geometry}
\usepackage{comment}

\newtheorem{theorem}{Theorem}

\newtheorem{lemma}{Lemma}
\newtheorem{corollary}{Corollary}

\newtheorem{definition}{Definition}

\begin{document}

\title{A Note on an Analytic Approach to the Problem of Finite Matroid Representability, The Cardinality of Sets of k-Independent
Vectors over Finite Fields and the Maximum Distance Separable Conjecture }
\author{ Steven B. Damelin\thanks{zbMATH Open, European Mathematical Society, Department of 
Mathematics, FIZ Karlsruhe Leibnitz Institute for Information Infrastructure, Franklinstr.11 10587 Berlin: steve.damelin@gmail.com} and Jeffrey Sun\thanks{Department of Economics, University of Toronto, Toronto, Canada; jeffreyesun@gmail.com}}
\maketitle

\begin{abstract}
We introduce various quantities that can be defined for an arbitrary finite matroid, and show that certain conditions on these quantities imply that a finite matroid is not representable over $\mathbb{F}_q$ where $q$ is a prime power.  Mostly, for a finite matroid of rank $r$, we examine the proportion of size-$(r-k)$ subsets that are dependent, and give bounds, in terms of the cardinality of the matroid and $q$, for this proportion, below which the matroid is not representable over $\mathbb{F}_q$. We also explore connections between the defined quantities and demonstrate that they can be used to prove that random matrices have high proportions of subsets of columns independent. Our study relates  to the results of our papers \cite{dmms, dmic, sun} dealing with the cardinality of sets of $k$-independent
vectors over $\mathbb{F}_q$ and the Maximal Distance Separation Conjecture over $\mathbb{F}_q$. 
\bigskip

\end{abstract}

\section{Introduction}

A finite matroid $M$ is a pair ($\Lambda$, $\tau$) where $\Lambda$ is a finite set and $\tau$ is a family of subsets of $\Lambda$ (called independent sets) with the folowing properties. 
(a) The empty set is independent, (b) every subset of an independent set is independent and (3) if $A$ and $B$ are two independent sets and $A$ has more elements than $B$, then there exists $x\in A\backslash B$ such that $B\cup \left\{x\right\}$ is independent. 
The rank of the matroid $M$  is defined as the maximum cardinality of an independent set in $\tau$.  For $q$ a prime power, let  $\mathbb{F}_q$  be  the finite field of order $q$.
In this note, we introduce various quantities that can be defined for an arbitrary finite matroid $M$ and we study its representability over the field $\mathbb{F}_q$ where $q$ is a prime power.  A matrix representation of a finite matroid is a way to encode its independence structure using a matrix over a field, so that the linear independence of the matrix columns corresponds exactly to the independence of the matroid’s elements.  More precisely, a representation of a finite matroid $M$ over  $\mathbb{F}_q$  is a matrix 
$A$ with entries in  $\mathbb{F}_q$  having one column for each element of $\Lambda$ such that a subset $I\subseteq \Lambda$ is independent in 
$M$ if and only if the corresponding columns of $A$ are linearly independent over  $\mathbb{F}_q$. For a finite matroid  of rank $r$, we examine the proportion of size-$(r-k)$ subsets that are dependent, and give bounds, in terms of the cardinality of the matroid and $q$, for this proportion, below which the matroid is not representable over $\mathbb{F}_q$. We also explore connections between the defined quantities and we demonstrate that they can be used to prove that random matrices have high proportions of subsets of columns independent. Our study relates  to the results of our papers \cite{dmms, dmic} dealing with the cardinality of sets of $k$-independent vectors over $\mathbb{F}_q$ and our recent result from our paper \cite{sun} on the Maximal Distance Separation Conjecture over $\mathbb{F}_q$. We briefly survey these results.

\section{The cardinality of sets of $k$-independent vectors over finite fields of primer order}

For $q$ a prime power and with  $\mathbb{F}_q$ denoting the finite field of order $q$, we let $F_q^n$ 
denote the $n$-dimensional vector space of all $n$-tuples over  $\mathbb{F}_q$.  Here, $n\geq 1$ is an integer. For an integer $k$ with $1\leq k\leq n$, a set of vectors $A\subset F_q^n$ is $k$-independent if all its subsets with at most $k$ elements are linearly independent.
An important quantity moving forward,  $Ind_q(n,k)$, is the cardinality of a $k$ independent subset of $F_q^n$. 
Our papers \cite{dmms, dmic} studied the cardinality of the quantity $Ind_q(n,k)$ for fixed $q,n,k$.
\subsection{An easy observation}
In examining the quantity $Ind_q(n,k)$, our first easy observation is the following:
\begin{itemize}
\item[(a)] $q^n-1=Ind_q(n,1)\geq Ind_q(n,2)\geq...\geq Ind_q(n,n)\geq n+1$ with equality with $q=2$.
\item[(b)] $Ind_q(n,2)=\frac{q^{n-1}}{q-1}$.
\end{itemize}

Indeed, (a) follows from the fact that any set of nonzero vectors is $1$-independent, $(k+1)$ independence implies $k$-independence and the $(n+1)-$ element set consisting of the standard basis plus the "all ones" vector is $n$-independent. 
(b) follows from the fact that two (nonzero) vectors are linearly independent if and only if neither is a multiple of the other.

.\section{Two extreme cases $k\leq 3$ and $k\geq 2n/3$}

In our paper \cite{dmms}, we investigated formulae for $Ind_2(n,k)$ in two extreme cases. $k\leq 3$ and $k\geq 2n/3$.
Indeed, our main result in \cite{dmms} is given in the following theorem:

\begin{theorem}
The following hold:
\begin{itemize}
\item[(a)] $Ind_2(n,3)=2^{n-1}$.
\item[(b)] $Ind_2(n,n-m)=n+1$ for $n\geq 3m+2,\, m\geq 0$.
\item[(c)] $Ind_2(n,n-m)=n+2$, for $n=3m+i,\, i=0,1,\, m\geq 2$.
\end{itemize}
\end{theorem}

\section{$Ind_q(n,k)=n+1$}

In our paper \cite{dmic}, we generalized part (b) of Theorem 1 proving a condition on $q,n,k$ which is both necessary and sufficient for $Ind_q(n,k)=n+1$ to hold. Indeed our main result in our paper \cite{dmic} is given by the following theorem:

\begin{theorem}
Assume that $2\leq k\leq n$. Then the following holds:
\begin{itemize}
\item[(a)] $Ind_q(n,k)=n+1$ if and only if $\frac{q(n+1)}{q+1}\leq k$.
\end{itemize}
\end{theorem}
One checks easily that (a) is consistent with Theorem 1(b).

\section{Analytic approach to the problem of finite matroid representability}

By a subset of a matrix we will mean a subset of its columns, and by its size we will mean the total number of columns it has. We will say a finite matroid is $q$-representable if it has a matrix representation over $\mathbb{F}_q$.

\subsection{A generalization of uniformity}
First, we generalize a basic definition.
\begin{definition}
A finite matroid $M$ of rank $r$ is said to be \textbf{uniform} if every size-$r$ subset of $M$ is independent.
\end{definition}
\begin{definition}
We define the \textbf{$k$-dependence} of a finite matroid of rank $r$ as the proportion of its size-$(r-k)$ subsets that are dependent. When a matrix has $k$-dependence $0$, we call it \textbf{$k$-independent}, otherwise we call it \textbf{$k$-dependent}. For a finite matroid $M$, we will denote rank by $r(M)$, cardinality by $s(M)$, and $k$-dependence by $d(M,k)$.
\end{definition}
Note that, by these definitions, a finite matroid $M$ is uniform if $d(M,0) = 0$, i.e., if it is $0$-independent.

\subsection{Optimal representable matrices}
It is natural to try to optimize some property of a finite matroid given certain constraints, especially $q$-representability. We use the following symbols to denote optimal achievable quantities:
\begin{definition}
$ $\newline
\begin{itemize}
\item By $Ind_q(r,k,d)$, we mean the largest $s$ such that there exists some full-rank $r \times s$ matrix $M$ over $\mathbb{F}_q$ with $k$-dependence $\leq d$. Equivalently, it is the size of the largest $q$-representable rank-$r$ matroid with $k$-dependence $\leq d$. 
\item By $D_q(r,k,s)$, we mean the smallest $d$ such that there exists some full-rank $r \times s$ matrix $M$ over $\mathbb{F}_q$ with $k$-dependence $\leq d$. Equivalently, it is the smallest $k$-dependence of any $q$-representable rank-$r$ matroid of size $s$.
\end{itemize}
\end{definition}
These quantities prove useful because we can use them to say the following:
\begin{lemma}
Let $M$ be a finite matroid. If, for some $k$,
\begin{itemize}
\item If $Ind_q(r(M),k,d(M,k)) \leq s(M)$
or
\item If $D_q(r(M),k,s(M)) \leq d(M,k)$ for some $k$,
\end{itemize}
then $M$ is not $q$-representable.
\end{lemma}

\section{Equivalences of bounds}
An equivalence between bounds on $Ind$ and on $D$ exist due to the following:
\begin{lemma}
As a function of $s$, $D_q(r,k,s)$ is non-decreasing.
\end{lemma}
\begin{proof}
Let $M$ be a minimally $k$-dependent $q$-representable matroid of size $s$. That is, because we are dealing with finite sets and infima are always achievable,
$$d(M) = D_q(r(M),k,s(M)).
$$
Then, for every finite matroid $M'$ obtained by deletion of one element from $M$,
$$
d(M') \geq D_q(r(M'),k,s(M')) = D_q(r(M)-1,k,s(M)-1).
$$
Thus, because each size-$(n-k)$ subset is counted an equal number of times in the measurement of the $d(M')$,
$$
D_q(r(M),k,s(M)) \geq D_q(r(M)-1,k,s(M)-1).
$$
\end{proof}
The equivalence between bounds can be stated thus:
\begin{lemma}
\begin{itemize}
\item If, for some $q,r,k,d, Ind_q(r,k,d) < s$, then, for any $s' \geq s$, it holds that
$$
D_q(r,k,s') > d.
$$
\item If, for some $q,r,k,s, D_q(r,k,s) > d$, then, for any $d' \leq d$, it holds that
$$
Ind_q(r,k,d') < s.
$$
\end{itemize}
\end{lemma}

\section{Explicit bounds}
We give various explicit bounds on $Ind$ and $D$, on whichever of the two the explanation of the bound is simplest. In each case, the equivalent statement on the other function is implied.
\begin{theorem}
$Ind_q(r,k,0) \leq q^{k+1}(r-k-1)$
\end{theorem}
\begin{proof}
Suppose some finite matroid $M$ is $q$-representable. Then some $r(M) \times s(M)$ matrix $M'$ over $\mathbb{F}_q$ can be constructed with all size-$(n-k)$ subsets independent.

We treat the columns of $M'$ as vectors in $\mathbb{F}_q^r$, and assume that none of them are the zero vector.

Observe that at most $r-k-1$ columns of $M'$ can lie within a $(r-k-1)$-plane. This implies that the proportion between $(r-k-1)$ and the number of points in an $(r-k-1)$-plane bounds the proportion of the total number of vectors in $\mathbb{F}_q^r$ that are represented as columns in $M'$.

Explicitly, this proportion is
$$
\frac{r-k-1}{q^{r-k-1}}
$$
out of
$$
q^r
$$
vectors in the space. Thus, the total number of columns is bounded by
$$
q^r\frac{r-k-1}{q^{r-k-1}} = q^{k+1}(r-k-1).
$$
\end{proof}

\begin{theorem}
For some integer $n$, $D_q\left(r,k,n\frac{q^n-1}{q-1}\right)$ is minimized by the matrix $M$ consisting of $n$ copies of each unique nonzero vector in $\mathbb{F}_q^r$ up to scaling.
\end{theorem}
That is, the matrix consists of exactly $n$ representatives of each point in the projective space.
\begin{proof}
Let $M$ as above. The claim is clearly true for $k = n-2$, in which case $M$ is the only vector of the required size that is $(n-2)$-dependent. We proceed inductively.
Let $M$ as above, $\vec{v} \in M$. We can view $M$ as a multiset of points in the projective space $\mathbb{P}^{r-1}(\mathbb{F}_q)$. Let $\mathbb{P}^{r-2}$ be some hyperplane in $\mathbb{P}^{r-1}(\mathbb{F}_q)$. Then a set $S$ including one copy of $\vec{v}$ is independent if and only if the projection of $S\setminus\{\vec{v}\}$ from $\vec{v}$ onto $\mathbb{P}^{r-2}$ is independent. A set containing two copies of $\vec{v}$ is dependent. Thus, removing all copies of $\vec{v}$ from $M$, we can count the number of independent size-$(n-k)$ sets containing $\vec{v}$ by counting the number of independent size-$(n-k-1)$ points of the projection of the remaining members of $M$ onto $\mathbb{P}^{r-2}$ as above. If $M$ contains one column for each vector in $\mathbb{F}_q^r$, then exactly $q+1$ vectors will be projected to each point in the hyperplane. The $(k-1)$-dependence for that arrangement of vectors in the hyperplane, by the inductive hypothesis, is optimal.
\end{proof}

\section{$k$-extensions}
By putting $k$-independent matrices in a certain form, we can obtain a necessary and sufficient condition for $k$-independence.
\begin{definition}
Let $M$ be an $r \times s$ matrix of full rank. Rearrange the columns so that the first $r$ are linearly independent, and form an invertible matrix $A$. Multiply $M$ by $A^{-1}$. None of these operations affect which subsets of $M$ are independent. Now, $M$ should be of the form $\lbrack Id\vert M'\rbrack$, where $Id$ is the identity matrix, and $M'$ is called the ``extension''. If $M$ is $k$-independent, we call $M'$ a $k$-extension.
\end{definition}
In many ways, it turns out to be easier to directly examine the extension of a matrix. For example,
\begin{theorem}
A matrix $M$ is a $k$-independent if and only if every $(n+k) \times n$ submatrix of its extension is full-rank, for every $n$.
\end{theorem}
\begin{proof}
First, let $M'$ be an $r \times (r+\ell)$ systematic $k$-dependent matrix constructed with extension $M$. Because it is $k$-dependent, there exists some set of $r-k$ columns of $M'$ that sum to zero. Let us say that $n$ of these columns are in the extension, and $(r-k)-n$ of them are in the identity part. Then the $n$ columns combine to a vector with zeroes where all the $k-n$ columns have zeroes, which is in $r-((r-k)-n) = k+n$ places. Take the submatrix of $M$ that consists of the $n$ columns and the $k+n$ rows that they combine to zero on. Then this submatrix does not have full rank because its columns are dependent.

Next, suspending any requirement on $M'$ other than its size and that the first $r$ columns form an identity matrix, let there exist some $(k+n) \times n$ submatrix of $M$ that is not full-rank. Then its columns combine to the zero vector. In that case, extending this submatrix to the full height $r$, the $n$ columns combine to a vector $\vec{v}$ with $n+k$ zeroes. Thus, these columns, together with the $r-(n+k)$ columns of the identity with ones where $\vec{v}$ does not necessarily have zeroes, combine to zero. These columns number $r-(n+k) + n = r-k$, and so $M'$ is $k$-dependent.
\end{proof}

For the case $k=0$, this was pointed out by \cite{bell}.
\medskip


Theorem 5 has several interesting implications:
\begin{itemize}

\item Every $r \times (r-k)$  submatrix of its extension is full-rank, meaning that any collection of $n-k$ parallel column $n$-vectors within the matrix are independent.

\item Every $(r+k) \times r$ submatrix $S$ of its extension contains an invertible matrix as a subset of its rows. This bounds the number of possible
sets of $r$ rows of S that are dependent. It also bounds the number of sets of $r$ rows of any submatrix of its extension that are dependent. Looking at this submatrix
sideways, this is of course a condition on dependent columns. 
\item For the $k=0$ case, where we are discussing extensions of a uniform matroid, the condition is that all submatrices are of full rank. Equivalently, it is that all square submatrices are invertible. This property is closed under transpose, and so given a uniform matrix matroid, putting it in systematic form (which is always possible), the dual matroid can be constructed by simply taking the transpose of the extension and appending a larger identity matrix. We check that this takes a matrix of size $n \times (n+l)$, with extension of size $n \times l$, and produces a matrix with extension size $l \times n$, of size $l \times (n+l)$.
\item Our arguments provide ways to produce  $(k+1)$-extensions out of $k$-extensions, possibly the simplest being the addition of a zero row to the $k$-extension. This gives inductive lower bounds on the maximal sizes of $k$-extensions.
This observation may indeed lead to methods for making $(k-1)$-extensions out of $k$-extensions.
\item Any submatrix of a $k$-extension is a $k$-extension and this fact may suggest ways to analyze the transposes of $k$-extensions. There are many equivalent ways to state this. For example, out of a matrix $M$, construct the block matrix $M'' = [M|0_k]$, where $0_k$ stands for the matrix with $k$ zero columns. Then $M$ is a $k$-extension if and only if every $r \times r$ submatrix of $M''$ has rank at least $r-k$.

\item If the MDS conjecture is true, (see Section 9 below), then the sizes of the $0$-extensions are precisely those that can fit in the top triangle of a $q \times q$ matrix. 

\item A study of the overlapping behavior of $k$-extensions could prove fruitful in the following sense. That is, whether there are $k$-extensions that overlap on some submatrix, or whether they are totally overlapping, meaning that they overlap on a submatrix with height the minimum of the heights of the two matrices, and the same with width. If there is sufficient overlapping of $k$-extensions, then a large matrix with a staircase-shaped nonzero region can be constructed, every rectangular submatrix of whose nonzero region is a $k$-extension. In the $k=0$ case, \cite{bell} presents such triangles for $p \in \{5,7\}$

\item Similarly, if you can start with an $n \times (q-n) 0-$extension which admits either a row addition or a column addition, the MDS conjecture (for that $n$) is equivalent to being unable to fill in the last entry in the corner not covered by the new row or column. This is because, without the final entry, every submatrix of the matrix with the new column and row is either a submatrix of the row-added matrix or the column-added matrix, so the row-added column-added matrix fulfils the submatrix-rank condition.

\item The theorem suggests a generalization of $k$-dependence to $k < 0$, in which the width of the required full-rank matrices exceed their height. Note that under the transpose, $k$-extensions are taken to $(-k)$-extensions. This is consistent with $0$-extensions being taken to $0$-extensions.

\end{itemize}

\section{Random matrices}

Using the quantities defined at the beginning of this paper, and defining two more, can be used to prove that, with very high probability, a very high proportion of the subsets of a certain size of a random matrix are independent.

By ``random matrix,'' we mean a matrix whose columns are randomly chosen nonzero vectors.

\begin{definition}
Let $Ind_q(r,k,d,p)$ be the largest $s,$ or $D_q(r,k,s,p)$ the smallest $d$, or $P_q(r,k,d,s)$ the smallest $p$, such that, with probability $1-p$, a random $r \times s$ matrix has $k$-dependence $\leq d$.
\end{definition}

\begin{lemma}
Denote the probability that $(r-k)$ nonzero vectors chosen randomly from $\mathbb{F}_q^r$ are independent by $\pi_{q,r,k}$. Then,
$$
\pi_{q,r,k}=\prod_{i=0}^{r-k}{\frac{q^r-q^i}{q^r-1}}.
$$
\end{lemma}
\begin{proof}
Each term of the product divides the number of points outside an $i$-plane by the number of nonzero points in the space. This is the probability that, given that we have already picked $i$ independent vectors, that the next one we pick will lie outside the span of those $i$.
\end{proof}

\begin{theorem}
$$
D_q(r,k,s) \leq 1 - \pi_{q,r,k}.
$$
\end{theorem}
\begin{proof}
Take a certain choice of $(r-k)$ distinct integers between $1$ and $r$. These correspond to a single size-$(r-k)$ subset of a matrix of size $s$. Then, the proportion of this particular subset of all size-$s$ matrices that are independent is equivalently $\pi_{q,r,k}$. Since this proportion is equal for any choice of subset, we have that the proportion of all size-$(r-k)$ subsets of all matrices of size $s$ is $\pi_{q,r,k}$. Thus, some matrix achieves this proportion.
\end{proof}
\begin{corollary}
$1-\pi_{q,r,k}$ is the mean $k$-dependence of all $r \times s$ matrices without zero columns.
\end{corollary}
Because $\pi_{q,r,k}$ is in general very close to one, viewing $p$ as a proportion of the set of all $r \times s$ matrices, we can get bounds on $D_q(r,k,s,p)$. Specifically,
\begin{theorem}
For any $q,r,k,s,p,$
$$
D_q(r,k,s,p) \leq \frac{1- \pi_{q,d,k}}{p}.
$$
\end{theorem}
\begin{corollary}
For any $q,r,k,s,d,$
$$
P_q(r,k,s,d) \leq \frac{1- \pi_{q,d,k}}{d}.
$$
\end{corollary}
Note that these quantities do not depend on $s$.

\section{Maximum distance separable (MDS) conjecture}

The work in this section is taken from our recent paper \cite{sun}. Suppose $2 \leq k \leq q$. A $k\times n$ maximum distance separable ($k\times n$ MDS) code $M$ is a $k \times n$ matrix with entries in $\mathbb{F}_q$ such that every set of $k$ columns of $M$ is linearly independent. The Maximum Distance Separable (MDS) conjecture is a well-known problem in coding theory and algebraic geometry with important consequences for example to the study of arcs in finite projective spaces and coding theory. See \cite{Ball, Ball1, sun} and the many references cited therein.

The conjecture, first posed by Singleton in 1964 \cite{Singleton} gives a possible upper-bound on the size of a $k\times n$ MDS code. 
More precisely, the MDS conjecture says the following:
\medskip

{\bf MDS Conjecture}: The maximum width, $n,$ of a $k \times n$ MDS code with entries in $\mathbb{F}_q$ is $q+1$, unless $q$ is even and $k \in \{3,q-1\}$, in which case the maximum width is $q+2$.
\medskip

We remark that there exist $k\times n$ MDS codes that attain the maximum possible width as given by the MDS conjecture. These are the Reed-Solomon codes (Definition 6 below),  see for example \cite{sun,KO,Lidl} and the many references cited therein.
\medskip

Henceforth all matrices will have entries in $\mathbb{F}_q$. When we speak to linear combinations, we mean nontrivial $\mathbb{F}_q$-linear combinations.

\subsection{Some known work on the MDS conjecture}

We provide some known (not exhaustive) work on the MDS Conjecture. An excellent survey on the MDS conjecture  can be found in \cite{Ball1}. The MDS conjecture has been verified when $q$ is prime. It is also known to hold when $q$ is a square and with $k\leq c\sqrt{pq}$ where the constant $c$ depends on whether $q$ is odd or even. When $q$ not a square, it is known to be true for $k\leq c'\sqrt{pq}$ where $c'$ depends on whether $q$ is odd or even.  It is also known to hold for all $k\times n$ MDS codes with alphabets of size at most 8.
\medskip

We are ready to state the main result of our recent paper \cite{sun}. which provides necessary and sufficient conditions for the MDS conjecture to hold. 

\subsection{Statement of the main result of \cite{sun}}

In this section, we state our main result, Theorem 8.  We need some definitions and notation. These are given in:

\begin{definition}
\begin{itemize}
\item[(a)] We denote by $\mathcal{P}_q$ the ring $\mathbb{F}_q[x]/(x^q-x)$ of polynomial functions over $\mathbb{F}_q$ of maximum degree $q$. $\mathcal{P}_q$ is a vector space over $\mathbb{F}_q$. 
\item[(b)] Throughout, the \textit{perp space} of a vector, $\vec{v}$ in an inner product space, $(\mathbf{V},\cdot):=\mathbf{V}$  is the set given by $\vec{v}^\perp = \{ \vec{w} \in \mathbf{V} | \vec{w} \cdot \vec{v} = 0 \}$.\\
	\indent The \textit{perp space} of a subspace, $\mathbf{U} \subseteq \mathbf{V}$, where $\mathbf{V}$ is an inner product space is the set given by $\mathbf{U}^\perp = \{ \vec{w} \in \mathbf{V} | \forall \vec{u} \in \mathbf{U},\vec{w} \cdot \vec{u} = 0 \}$. ($\vec{w}$ and $\vec{u}$ are orthogonal). 
\item[(c)] For every non-negative integer $n$, we define the subset $\mathcal{O}_n \subset \mathcal{P}_q$ as the set of polynomials in $\mathcal{P}_q$ that are either the zero polynomial, or have at most $n$ distinct roots in $\mathbb{F}_q$. If $n \geq q$, then $\mathcal{O}_n = \mathcal{P}_q$.
\item[(d)] We denote by $\langle Y, Z \rangle$ the subspace of $\mathcal{P}_q$ generated by the elements of subspaces $Y$ and $Z$.
\item[(e)] A \textit{Reed-Solomon code} of dimension $k$ is a $k \times q$ matrix with entries in $\mathbb{F}_q$ such that each column of the matrix is of the form $[1,a,a^2,\ldots,a^{k-2},a^{k-1}]^\intercal$ for some $a \in \mathbb{F}_q$.
\item[(f)] An \textit{Extended Reed-Solomon code} is a Reed-Solomon code with the column $[0,0,0,\ldots,0,1]^\intercal$ appended.
\end{itemize}
\end{definition}

Both Reed-Solomon and Extended Reed-Solomon codes are $k\times n$ MDS codes and it proved in \cite{Seroussi}, that for odd $q$,  no column other than $[0,0,0,\ldots,0,1]^\intercal$ can be appended to a Reed-Solomon code to produce another $k\times n$ MDS code.\\
\medskip

The main result in our paper \cite{sun} is given as Theorem 8 below:
\begin{theorem}

Let k be an integer such that $2 \leq k \leq q = p^r,$ where p is prime and r is a positive integer. Suppose that either
		\begin{enumerate}
			\item $q$ odd
			\item $q$ even and $k\in \{3, q-1\}$.
		\end{enumerate}
		
Consider the following statements (1-5) below:

\begin{enumerate}		

\item[(1)] The MDS Conjecture is true. That is, the maximum width, $n,$ of a $k \times n$ MDS code with entries in $\mathbb{F}_q$ is $q+1$, unless $q$ is even and $k \in \{3,q-1\}$, in which case the maximum width is $q+2$.

\item[(2)] Let $M'$ be a $k \times (q+2)$ matrix. Then some linear combination of the rows of $M'$ has at least $k$ zero entries.
		
\item[(3)] Let $M'$ be a $k \times (q+2)$ matrix such that the first two columns of $M'$ are $[1,0,\ldots,0]^\intercal$ and $[0,1,0,\ldots,0]^\intercal$. Then some linear combination of the rows of $M'$ has at least $k$ zero entries.

\item[(4)] There do not exist distinct subspaces $Y$ and $Z$ of $\mathcal{P}_q$ such that
		\begin{enumerate}
			\item $dim(\langle Y, Z \rangle) = k$.
			\item $dim(Y) = dim(Z) = k-1$.
			\item $\langle Y, Z \rangle \subset \mathcal{O}_{k-1}$
			\item $Y\cup Z \subset \mathcal{O}_{k-2}$
			\item $Y\cap Z \subset \mathcal{O}_{k-3}$.
		\end{enumerate}
		
\item[(5)] There is no integer $s$ with $k < s \leq q$ such that the Reed-Solomon code $\mathcal{R}$ with entries in $\mathbb{F}_q$ of dimension $s$ can have $s-k+2$ columns $\mathcal{B} = \{b_1,\ldots,b_{s-k+2}\}$ added to it, such that:
		\begin{enumerate}
			\item Any $s \times s$ submatrix of $\mathcal{R} \cup \mathcal{B}$ containing the first $s-k$ columns of $\mathcal{B}$ is independent (non-zero determinant).
			\item $\mathcal{B} \cup \{[0,0,\ldots,0,1]^{\intercal} \}$ is independent.
		\end{enumerate}
		
	\end{enumerate}
\bigskip	

Then the following holds true:

\begin{itemize}
\item[(Part A)]:  (1), (2), (3) and (5) are equivalent.
\item[(Part B)]:  (3) implies (4).
\end{itemize}
\end{theorem}

\section{Funding Declaration}

The authors would like to gratefully acknowledge the support of the National Science Foundation.

\section{Conflicts of Interests}

The authors declare no conflicts of interest.

\end{document}